\documentclass[12pt,oneside,reqno]{amsart}
\usepackage{amsmath,amsfonts,amssymb,stmaryrd}

\usepackage{color}
\usepackage[T1]{fontenc}
\usepackage{lmodern} 
\usepackage[utf8]{inputenc}   

\usepackage{tikz}
\usetikzlibrary{calc}
\usepackage{subfig}

\newtheorem{theorem}{Theorem}[section]

\newtheorem{lemma}[theorem]{Lemma}

\theoremstyle{definition}
\newtheorem{definition}[theorem]{Definition}
\theoremstyle{remark}
\newtheorem{remark}[theorem]{Remark}

\numberwithin{equation}{section}
\def\u{{\bf u}}
\def\O{\Omega}

\def\R{\mathbb{R}}
\def\r{\mathbb{R}}

\def\ve{\varepsilon}

\def\u{{\bf u}}
\def\vv{{\bf v}}
\def\O{\Omega}

\def\R{\mathbb{R}}

\def\ve{\varepsilon}

\begin{document}

	\date{}

\title{Solutions of the divergence equation in Lipschitz spaces}

\author{Mar\'\i a E. Cejas}
\address{IMAS (UBA-CONICET) 
	Ciudad Universitaria\\
	(1428) Ciudad Aut\'onoma de Buenos Aires\\  
	Argentina;
	e-mail: ecejas@mate.unlp.edu.ar}

\author{Ricardo G. Dur\'an}
\address{IMAS (UBA-CONICET) and Departamento de Matem\'atica\\
Facultad de Ciencias Exactas y Naturales\\
Universidad de Buenos Aires\\
Ciudad Universitaria\\
(1428) Ciudad Aut\'onoma de Buenos Aires\\
Argentina;	e-mail: rduran@dm.uba.ar.}



\begin{abstract}
	We study the solvability of the divergence equation
	\[
	\operatorname{div} \u = f
	\]
	in bounded $C^2$ domains under homogeneous Dirichlet boundary conditions for data
	$f\in C^{0,\alpha}(\Omega)$
	satisfying the compatibility condition
	$
	\int_\Omega f =0.
	$
	We construct a solution $\u$ such that for every $0<\beta<\alpha$
	$$
	\u\in C^{1,\beta}(\Omega)^n
	$$
	satisfies
	$$
	\|\u\|_{C^{1,\beta}(\Omega)}
	\le
	C\|f\|_{C^{0,\alpha}(\Omega)}.
	$$
	The proof combines localization techniques with a boundary flattening procedure reducing the problem to a model half-cube.
	
\end{abstract}

\maketitle
\section{Introduction}

Let $\Omega \subset \mathbb{R}^n$ be a bounded domain. The problem of solving the divergence equation
\begin{equation}\label{eq:div_intro}
	\operatorname{div} \u = f \quad \text{in } \Omega,
\end{equation}
under the compatibility condition
\[
\int_\Omega f =0,
\]
plays a fundamental role in the analysis of partial differential equations. The construction of right inverses of the divergence operator is a key ingredient in the study of incompressible fluids, particularly for the Stokes and Navier-Stokes systems, and is also closely related to Korn-type inequalities.

For $1<p<\infty$, the classical result of Bogovski\v{\i}  provides a bounded linear operator mapping $L^p_0(\Omega)$ into $W^{1,p}_0(\Omega)^n$, yielding solutions of \eqref{eq:div_intro} together with quantitative norm estimates, see \cite{B}. The original construction was obtained in star-shaped domains through an explicit integral formula and extended to bounded Lipschitz domains by decomposing them as the union of a finite number of star-shaped domains. Since then, Bogovski\v{\i}-type operators have become a standard tool in fluid mechanics and PDEs, see also \cite{DM,CM,AD}.

In the endpoint cases $p=1$ and $p=\infty$, solvability fails in the classical Sobolev framework (see \cite{DFT}), which motivates the search for alternative functional settings. In the previous work \cite{CD}, we studied the divergence equation both in Hardy--Sobolev spaces and in Lipschitz spaces. In particular, for compactly supported data $f\in C^{0,\alpha}(\mathbb{R}^n)$, $0<\alpha<1$,
with vanishing mean, we obtained solutions together with global regularity estimates in Lipschitz spaces.

The compact support assumption plays an essential role in the approach developed in \cite{CD}. Indeed, the construction there produces solutions defined in the whole space $\mathbb{R}^n$ whose regularity is obtained through singular integral estimates. The boundedness of singular integral operators in Lipschitz spaces is proved in \cite{M,St,K2}. Moreover, the solutions produced by this method vanish outside the support of the data, and this property is crucial in order to obtain global Lipschitz estimates. 

This approach does not directly extend to arbitrary $f$. This naturally leads to the problem of whether one can construct solutions with first derivatives in Lipschitz spaces for general $f\in C^{0,\alpha}(\Omega)$ satisfying the compatibility condition.

Some related results in this direction were obtained by Berselli and Longo \cite{BL}, who proved the existence of $C^1(\overline{\Omega})$ solutions under a Dini continuity assumption on the data. They also observed that their methods could likely be adapted to Lipschitz spaces. However, such an extension does not seem to follow directly from their arguments, mainly because the treatment of boundary corrections in the Lipschitz framework requires additional care.

The purpose of this paper is to establish solvability of (\ref{eq:div_intro}) in Lipschitz spaces for bounded $C^2$ domains under homogeneous Dirichlet boundary conditions. More precisely, given
$f\in C^{0,\alpha}(\Omega)$, $0<\alpha<1$, satisfying $\int_\Omega f =0$,
we prove that  there exists a solution $\u$ of the divergence equation  (\ref{eq:div_intro} ), such that for every $0<\beta<\alpha$, $\u \in C^{1,\beta}(\Omega)^n$
and
$$
\|\u\|_{C^{1,\beta}(\Omega)}
\le
C\|f\|_{C^{0,\alpha}(\Omega)}.
$$

Our approach combines the compactly supported solvability theory developed in \cite{CD} with localization and boundary techniques inspired by \cite{BL}. Interior estimates are obtained directly from the compact support theory established in \cite{CD}, avoiding the explicit construction of the solution operator used in \cite{BL}. This yields the desired interior regularity estimates in a more direct way and allows the analysis to focus on the behavior near the boundary.

Near the boundary, we reduce the equation to a model problem in the half-cube setting, see \cite{BL}. The main difficulty then consists in constructing divergence-free boundary corrections compatible with the homogeneous Dirichlet condition while preserving Lipschitz regularity.

Although the data belongs to $C^{0,\alpha}(\Omega)$, the resulting solution satisfies estimates only in $C^{1,\beta}(\Omega)$ for every $0<\beta<\alpha$. This loss arises because the mollification procedure used in the boundary correction does not preserve the full Lipschitz exponent.

To the best of our knowledge, this is the first direct proof of solvability for the divergence equation in Lipschitz spaces with homogeneous Dirichlet boundary conditions for bounded domains.

The paper is organized as follows. In Section 2 we state the main result and introduce the functional setting. Section 3 is devoted to interior estimates. In Section 4 we construct solutions in a half-cube and develop the boundary correction. Finally, in Section 5 we combine these ingredients to obtain the global result.

\section{Functional setting and main result}
We begin by introducing the notation and function spaces used throughout the paper.
\begin{definition}
	Let $\Omega \subset \mathbb{R}^n$ be a bounded domain and $0<\alpha <1$.
	
	We define
	\[
	C^{0,\alpha}(\Omega) := \left\{ f \in C(\overline{\Omega}) :
	\sup_{\substack{x,y \in \Omega \\ x \neq y}}
	\frac{|f(x)-f(y)|}{|x-y|^\alpha} < \infty \right\}.
	\]
	We equip this space with the norm
	\[
	\|f\|_{C^{0,\alpha}(\Omega)} :=
	\|f\|_{L^\infty(\Omega)} +
	\sup_{\substack{x,y \in \Omega \\ x \neq y}}
	\frac{|f(x)-f(y)|}{|x-y|^\alpha}.
	\]
	
	Let $k \in \mathbb{N}$, we say that $f \in C^{k,\alpha}(\Omega)$ if
	$f \in C^k(\Omega)$, all derivatives $D^\gamma f$ with $|\gamma|\le k$
	extend continuously to $\overline{\Omega}$, and for every multi-index
	$\gamma$ with $|\gamma|=k$,
	$
	D^\gamma f \in C^{0,\alpha}(\Omega).
	$
	
	The norm is given by
	$$
	\|f\|_{C^{k,\alpha}(\Omega)} :=
	\sum_{|\gamma|\le k} \|D^\gamma f\|_{L^\infty(\Omega)} +
	\sum_{|\gamma|=k} [D^\gamma f]_{C^{0,\alpha}(\Omega)},
	$$
	where
	$$
	[D^\gamma f]_{C^{0,\alpha}(\Omega)} :=
	\sup_{\substack{x,y \in \Omega \\ x \neq y}}
	\frac{|D^\gamma f(x)-D^\gamma f(y)|}{|x-y|^\alpha}.
	$$

	If $\u:\overline{\Omega}\to\mathbb{R}^m$, we say that $\u \in C^{k,\alpha}(\Omega)^m$, if each component is in $C^{k,\alpha}(\Omega)$.
\end{definition}

By $C$ we will denote a generic constant which can change its value even in the same line.

With the above notation, we can now state the main theorem.

\begin{theorem}\label{thm:global}
	Let $\Omega \subset \mathbb{R}^n$ be a bounded $C^2$ domain and let $f\in C^{0,\alpha}(\Omega)$, with $0<\alpha< 1$, and $\int_\Omega f = 0$.
	Then, there exists $\u$ such that for every $0<\beta<\alpha$, $\u\in C^{1,\beta}(\Omega)^n$,
	$$
	\operatorname{div}\u=f \qquad \text{in }\Omega,
	$$
	$$
	\u=0 \qquad \text{on }\partial\Omega,
	$$
	and
	$$
	\|\u\|_{C^{1,\beta}(\Omega)}
	\le C \|f\|_{C^{0,\alpha}(\Omega)},
	$$
	where the constant $C$ depends on $\Omega$, $n$, $\alpha$, and $\beta$.
\end{theorem}

\begin{remark}
	The restriction $\beta < \alpha$ reflects a loss of regularity in the construction of the solution. This loss is inherent to the method, particularly due to the boundary correction procedure, which involves mollification arguments that do not preserve the full Lipschitz exponent. At present, it is not known whether the estimate can be improved to obtain solutions in $C^{1,\alpha}(\Omega)n$	under the same assumptions.
\end{remark}

To prove Theorem 2.2, we first establish interior estimates for the Bogovski\v{\i} operator. We then derive regularity estimates up to the boundary by combining the half-cube construction with a suitable change of variables.

\section{Interior estimates} \label{interiores}

In this section we establish the interior solvability estimates that will be used later in the localization argument for the global problem. The main point is that, away from the boundary, the divergence equation can be treated by means of the compactly supported solvability theory developed in \cite{CD}. This provides solutions with full Lipschitz regularity in interior subdomains and serves as the starting point for the boundary analysis carried out in the following sections.

Since the following estimates are interior in nature, we may assume without loss of generality that the domain is star-shaped. The general case follows by standard localization arguments.
 \begin{theorem}
 Let $0<\alpha<1$. Let $\u$ be the Bogovski\v{\i} solution   of 
 $$
 \operatorname{div}\u=f \qquad \text{in }\Omega,
 $$
 $$
 \u=0 \qquad \text{on }\partial\Omega,
 $$
for $f\in C^{0,\alpha}(\Omega)$ satisfying $\int_{\Omega}f=0$, $\u$ belongs to $C^{1,\alpha}(K)$ for any $K\subset \Omega$ compact set, and
$$
\|\u\|_{C^{1,\alpha}(K)}
\le C \|f\|_{C^{0,\alpha}(\Omega)}.
$$
where $C$ depends on $\Omega$ and $K$.
 \end{theorem}

\begin{proof}
We have that $\u$ is given by
$$
\u(x)=\int_\O G(x,y)f(y)\,dy
$$
with
$$
G(x,y)=\int_0^1 \left(\frac{x-y}{s}\right)\omega\left( y+\frac{x-y}{s}\right) \frac{\,ds}{s^n}
$$
where $\omega \in C_0^{\infty}(B)$, $B\subset \Omega$ a ball and $\int_B \omega =1$, for more details see \cite{CD, AD}.

Let $K\subset\O$ be a compact set, denote $d(x)$ the distance from $x$ to  $\partial \Omega$. Set $\ve>0$ so that $d(x)>\ve$ for every $x\in K$ and $\psi\in C_0^\infty(\O)$ such that $\psi(x)=1$ if $d(x)>\ve/2$.

Then,
$$
\u(x)=\u_1(x)+\u_2(x)=\int_\O G(x,y)\psi(y)f(y)\,dy
+\int_\O G(x,y) (1-\psi(y))f(y)\,dy
$$

Since $\psi f$ is compactly supported in $\Omega$ and can be extended by zero preserving its $C^{0,\alpha}$ regularity, we may apply \cite[Theorem 5.4]{CD} to obtain
$$
\|\partial_{x_j}\u_1\|
_{C^{0,\alpha}(\R^n)}
\le C\|\psi f\|_{C^{0,\alpha}(\R^n)} 
\le C\|f\|_{C^{0,\alpha}(\O)} 
$$
where C depends on $\psi$ and, consequently, on K.

On the other hand, if $x\in K$ and $y\in \mbox{supp\,}(1-\psi)$, we have that 
$$
\ve<d(x)\le |x-y|+d(y)<|x-y|+\ve/2
$$
as a consequence $|x-y|>\ve/2$. Then,
$$
\partial_{x_j}\u_2(x)=\int_\O\partial_{x_j}G(x,y) (1-\psi(y))f(y)\,dy.
$$
Moreover, the kernel vanishes unless
$s > \frac{|x-y|}{d}$,
where $d$ denotes the diameter of $\Omega$, see \cite[Lemma 2.1]{AD}. Since $|x-y| \geq \varepsilon/2$, we obtain $s > \frac{\varepsilon}{2d}$.
Therefore,
\[
\partial_{x_j}G(x,y)
=
\int_{\varepsilon/(2d)}^1
\partial_{x_j}
\left(
\frac{x-y}{s^n}
\omega\left(y+\frac{x-y}{s}\right)
\right)\, ds.
\]

 Given that $\omega$ is smooth and compactly supported, the kernel $\partial_{x_j} G(x,y)$ is smooth away from the singularity. Moreover, in the present situation we have $|x - y| \geq \varepsilon/2$, so the kernel and all its derivatives are uniformly bounded. Therefore, standard estimates for integral operators with smooth kernels imply that $\partial_{x_j} \u_2$ is Lipschitz continuous of order $\alpha$ in $K$, with
$$
\|\partial_{x_j} \u_2\|_{C^{0,\alpha}(K)} \leq C \|f\|_{C^{0,\alpha}(\Omega)}.
$$
where $C$ depends on $\Omega$ and                                                                                                                                             $K$.

It remains to estimate the $L^\infty$ norm of the solution $\u$.
 This follows directly from the integral representation of the Bogovski\v{\i} operator. Indeed, by \cite[Lemma 2.8]{AD} the kernel satisfies
$$
|G(x,y)|
\le
C|x-y|^{-n+1},
$$
and by \cite[Pag 29]{AD}, we have
$$
\|\u\|_{C^{\alpha,0}(\Omega)}
\le
C
\|f\|_{C^{\alpha,0}(\Omega)}.
$$
As a consequence, we have proved
$$
\|\u\|_{C^{1,\alpha}(K)}
\le C \|f\|_{C^{0,\alpha}(\Omega)}.
$$
\end{proof}

\section{The half-cube construction}

In this section we analyze the divergence equation in a model half-cube and develop the boundary correction procedure underlying the proof of the main theorem. The resulting local construction will later be combined with localization and flattening arguments to obtain the global estimates. 
We denote by
$$
Q_1:=(-1,1)^n,
\qquad
Q_1^+:=Q_1\cap\{x_n>0\}.
$$

We start with a convergence lemma in Lipschitz spaces for exponents $\beta<\alpha$. This result will play a key role in the construction of the boundary correction.

\begin{lemma}\label{lem:convholder}
	Let $0<\beta<\alpha<1$, let $g\in C^{0,\alpha}(\mathbb{R}^{n-1})$, and let $\rho\in C_0^\infty(\mathbb{R}^{n-1})$. Set $L:=\displaystyle\int_{\mathbb{R}^{n-1}}\rho(z')\,dz'$. For $\varepsilon>0$, define
	$$
	g^\varepsilon(x')
	:=
	\int_{\mathbb{R}^{n-1}}\rho(z')\,g(x'-\varepsilon z')\,dz',
	\qquad x'\in \mathbb{R}^{n-1}.
	$$
	Then
	$$
	g^\varepsilon \to Lg
	\qquad \text{in } C^{0,\beta}(\mathbb{R}^{n-1})
	\quad \text{as } \varepsilon\to 0.
	$$
	More precisely, there exists a constant $C=C(\alpha,\beta,n,\rho)$ such that
	$$
	\|g^\varepsilon-Lg\|_{C^{0,\beta}(\mathbb{R}^{n-1})}
	\le
	C\,\varepsilon^{\alpha-\beta}\,[g]_{C^{0,\alpha}(\mathbb{R}^{n-1})}.
	$$
	
\end{lemma}

\begin{proof}
	For every $x'\in \mathbb{R}^{n-1}$, we have
    $$
	|g^\varepsilon(x')-Lg(x')|
	\le
	\int_{\mathbb{R}^{n-1}}|\rho(z')|
	\,|g(x'-\varepsilon z')-g(x')|\,dz'.
	$$
	Since $g\in C^{0,\alpha}(\mathbb{R}^{n-1})$, 
	$$
	|g(x'-\varepsilon z')-g(x')|
	\le
	[g]_{C^{0,\alpha}(\mathbb{R}^{n-1})}\,\varepsilon^\alpha |z'|^\alpha.
	$$
	Hence
	$$
	|g^\varepsilon(x')-Lg(x')|
	\le
	[g]_{C^{0,\alpha}(\mathbb{R}^{n-1})}\,\varepsilon^\alpha
	\int_{\mathbb{R}^{n-1}}|\rho(z')|\,|z'|^\alpha\,dz'.
	$$
	Taking the supremum over $\mathbb{R}^{n-1}$, we obtain
	\begin{equation}\label{eq:linfty-conv-lemma41}
		\|g^\varepsilon-Lg\|_{L^\infty(\mathbb{R}^{n-1})}
		\le
		C\,\varepsilon^\alpha [g]_{C^{0,\alpha}(\mathbb{R}^{n-1})}.
	\end{equation}
	
	We now estimate the $C^{0,\beta}$ seminorm. Let $x',y'\in \mathbb{R}^{n-1}$, $x'\neq y'$. We distinguish two cases.
	
	\medskip
	
	\noindent
	\textbf{1. $|x'-y'|\ge \varepsilon$.}
	
	Using \eqref{eq:linfty-conv-lemma41}, we have
	$$
	|(g^\varepsilon-Lg)(x')-(g^\varepsilon-Lg)(y')|
	\le
	2\|g^\varepsilon-Lg\|_{L^\infty(\mathbb{R}^{n-1})}
	\le
	C\,\varepsilon^\alpha [g]_{C^{0,\alpha}(\mathbb{R}^{n-1})}.
	$$
	Therefore,
	$$
	\frac{|(g^\varepsilon-Lg)(x')-(g^\varepsilon-Lg)(y')|}{|x'-y'|^\beta}
	\le
	C\,\varepsilon^\alpha |x'-y'|^{-\beta}[g]_{C^{0,\alpha}(\mathbb{R}^{n-1})} \le C\,\varepsilon^{\alpha-\beta}[g]_{C^{0,\alpha}(\mathbb{R}^{n-1})}.$$
	
	\medskip
	
	\noindent
	\textbf{2. $|x'-y'|< \varepsilon$.}
	
	We have

	\[
	|(g^\varepsilon-Lg)(x')-(g^\varepsilon-Lg)(y')|
	\le
	|g^\varepsilon(x')-g^\varepsilon(y')|
	+
	|L|\,|g(x')-g(y')|.
	\]
	Since $g\in C^{0,\alpha}(\mathbb{R}^{n-1})$, we only need to estimate the first term,

	$$
	|g^\varepsilon(x')-g^\varepsilon(y')|
	\le
	\int_{\mathbb{R}^{n-1}}|\rho(z')|
	\,|g(x'-\varepsilon z')-g(y'-\varepsilon z')|\,dz'.
	$$
Hence,
	$$
	|g^\varepsilon(x')-g^\varepsilon(y')|
	\le
	[g]_{C^{0,\alpha}(\mathbb{R}^{n-1})}|x'-y'|^\alpha
	\int_{\mathbb{R}^{n-1}}|\rho(z')|\,dz'
	\le
	C [g]_{C^{0,\alpha}(\mathbb{R}^{n-1})}|x'-y'|^\alpha.
	$$
Then,
	$$
	\frac{|(g^\varepsilon-Lg)(x')-(g^\varepsilon-Lg)(y')|}{|x'-y'|^\beta}
	\le
	C [g]_{C^{0,\alpha}(\mathbb{R}^{n-1})}|x'-y'|^{\alpha-\beta}
	 \leq 	\,\varepsilon^{\alpha-\beta}[g]_{C^{0,\alpha}(\mathbb{R}^{n-1})}.$$
	
As a consequence,
	$$
	[g^\varepsilon-Lg]_{C^{0,\beta}(\mathbb{R}^{n-1})}
	\le
	C\,\varepsilon^{\alpha-\beta}[g]_{C^{0,\alpha}(\mathbb{R}^{n-1})}.
	$$
	Together with \eqref{eq:linfty-conv-lemma41}, this yields
	\[
	\|g^\varepsilon-Lg\|_{C^{0,\beta}(\mathbb{R}^{n-1})}
	=
	\|g^\varepsilon-Lg\|_{L^\infty(\mathbb{R}^{n-1})}
	+
	[g^\varepsilon-Lg]_{C^{0,\beta}(\mathbb{R}^{n-1})}
	\le
	C\,\varepsilon^{\alpha-\beta}[g]_{C^{0,\alpha}(\mathbb{R}^{n-1})}.
	\]
	This proves the claim.
\end{proof}

Now, we will construct a solution in a half-cube.

\begin{theorem}\label{thm:half-cube}
	Let $0<\beta<\alpha<1$, and let $f\in C^{0,\alpha}(Q_1^+)$ satisfy
	$\displaystyle\int_{Q_1^+} f(x)\,dx =0$
	and $\mbox{ supp }f\subset Q^+_{1/2}$. Then, there exists a vector field
	$
	\u\in C^{1,\beta}(Q_1^+)^n
	$
	such that
	$$
	\operatorname{div}\u=f \quad \text{in } Q_1^+,
	\qquad
	\u=0 \quad \text{on } \partial Q_1^+,
	$$
	and
	$$
	\|\u\|_{C^{1,\beta}(Q_1^+)}
	\le
	C\,\|f\|_{C^{0,\alpha}(Q_1^+)},
	$$
	where $C=C(n,\alpha,\beta)$.
\end{theorem}
\begin{proof}
	We follow the strategy from \cite{BL}, adapted to our setting. Starting from $f\in C^{0,\alpha}(Q_1^+)$ with zero mean, consider

	\[f^*(x) = \left\{ \begin{array}{lr} f(x) &  x \in \overline{Q_1^+}\\ f(x^*) &  x \in \overline{Q_1^-} \end{array} \right. \]
	where $x^*=(x',-x_n)$, $x'=(x_1,\dots, x_{n-1})$. Then,
	$f^*\in C^{0,\alpha}(Q_1)$ and has zero integral in $Q_1$. Indeed, the symmetry preserves the Lipschitz seminorm across the hyperplane $\{x_n=0\}$. According to \cite{CD}, since $\mbox{ supp }f^*\subset Q_{1/2}$ we can solve
	\[ \left\{ \begin{array}{lr} \mbox{ div } \u^* =f^* &  \mbox{ in } Q_1\\ \u^*=0  &  \mbox{ on } \partial Q_1 \end{array} \right. \]

	Define
	$$\tilde{\u}=\frac{1}{2}(u^*_1(x)+u_1^*(x^*),\dots,u^*_{n-1}(x)+u^*_{n-1}(x^*), u^*_n(x)-u_n^*(x^*)).$$

	Restrincting $\tilde{\u}$ to $Q_1^+$, we obtain that $\mbox{ div }\tilde{\u}=f$. 
	By construction, the normal component of $\tilde{\u}$ vanishes on $\{x_n = 0\}$, while the tangential components are even functions with respect to $x_n$. As a consequence, $\tilde{\u}$ does not necessarily satisfy the homogeneous boundary condition, which motivates the introduction of the correction term $\Psi$. Namely, we will construct a divergence-free
	correction $\Psi$ such that
	$$
	\Psi_j(x',0)=\widetilde u_j(x',0), \qquad j=1,\dots,n-1,
	\qquad
	\Psi_n(x',0)=0.
	$$
 Once such a correction has been constructed, we define
	$$
	\u:=\widetilde \u-\Psi
	$$
	then,
	$$
	\operatorname{div}\u=f \quad\text{in }Q_1^+,
	\qquad
	\u=0 \quad\text{on }\partial Q_1^+.
	$$
	
	The construction of $\Psi$ and the corresponding estimates are given in Theorem \ref{thm:boundary-correction}, which can be regarded as a Lipschitz-space analogue of \cite[Lemma 22]{BL}.

\end{proof}

\begin{remark}
	The previous result is stated for the unit half-cube only for simplicity of notation. By a standard rescaling argument, the same conclusion holds for arbitrary half-cubes of the form
	$$
	Q^+:=(-r,r)^n\cap\{x_n>0\},
	$$
	with constants depending only on $r$, $\alpha$, $\beta$, and the dimension.
\end{remark}

By extending $\tilde{u}_j(x',0)$ by zero to the whole subspace $\{x_n=0\}$, we may reduce the problem to the case $\mbox{ supp }\tilde{u}_j(x',0) \subset (-1/2,1/2)^{n-1}$. 

\begin{theorem}\label{thm:boundary-correction}
	Let $0<\beta<\alpha<1$. Let $\rho\in C_0^\infty(\mathbb{R}^{n-1})$ be supported in $B(0,1)$ and satisfy $\displaystyle\int_{\mathbb{R}^{n-1}} \rho(x')\,dx' =1$. Let $\theta\in C^2([0,\infty))$ satisfy
	$\theta(0)=1$, $\theta(t)=0 \quad \text{for } t\ge \frac12$.
	Let $\widetilde \u\in C^{1,\alpha}(Q_1^+)^n$ defined as above, and for $j=1,\dots,n-1$, define
	$$
	g_j(x'):=\widetilde u_j(x',0),
	$$
	and
	$$
	\varphi_j(x',x_n)
	:=
	x_n\,\theta(x_n)\,
	(g_j*\rho_{x_n})(x'),
	\qquad x=(x',x_n)\in Q_1^+,
	$$
	where
	\[
	\rho_{x_n}(z'):=x_n^{-(n-1)}\rho\!\left(\frac{z'}{x_n}\right).
	\]
	Define
	$$
	\Psi(x)
	:=
	\left(
	\partial_{x_n}\varphi_1(x),\dots,\partial_{x_n}\varphi_{n-1}(x),
	-\sum_{j=1}^{n-1}\partial_{x_j}\varphi_j(x)
	\right).
	$$
	Then $\Psi\in C^{1,\beta}(Q_1^+)^n$, $\operatorname{div}\Psi=0$ in $Q_1^+$, and
	$$
	\Psi_j(x',0)=\widetilde u_j(x',0),\qquad j=1,\dots,n-1,
	\qquad
	\Psi_n(x',0)=0.
	$$
	Moreover, $\Psi\equiv 0 \qquad \text{for } x_n\ge \frac12$, and
	$$
	\|\Psi\|_{C^{1,\beta}(Q_1^+)}
	\le
	C\,\|\widetilde \u\|_{C^{1,\alpha}(Q_1^+)},
	$$
	where $C=C(n,\alpha,\beta,\rho,\theta)$.

\end{theorem}

\begin{proof}
	To verify that $\Psi\in C^{1,\beta}(Q_1^+)^n$, we need to prove that first and second derivatives of $\varphi_j$ are in $C^{0,\beta}(Q_1^+)$. Note that, since $\widetilde u_j(x',0) \in C^{1,\alpha}(\mathbb{R}^{n-1})$ and $\rho \in C^\infty_0(\mathbb{R}^{n-1})$, then the kernels $\rho_{x_n}$ and their derivatives with respect to both $x'$ and $x_n$ are smooth and compactly supported, uniformly for $x_n \in (0,1)$. Moreover, all the resulting integrands are bounded by integrable functions independent of $x$. Therefore, differentiation under the integral sign is justified by standard dominated convergence arguments.
	
	\medskip
	\noindent
	\textbf{First derivatives of $\varphi_j$.}
	
	For $j=1,\dots,n-1$, set
	$$
	g_j(x'):=\widetilde u_j(x',0), \qquad x'\in (-1,1)^{n-1}.
	$$
	Since $\widetilde \u\in C^{1,\alpha}(Q_1^+)^n$, we have
	$$
	g_j\in C^{1,\alpha}\bigl((-1,1)^{n-1}\bigr),
	\qquad
	\|g_j\|_{C^{1,\alpha}}
	\le
	C\|\widetilde \u\|_{C^{1,\alpha}(Q_1^+)}.
	$$
	
	By definition,
	$$
	\varphi_j(x',x_n)=x_n\theta(x_n)\,(g_j*\rho_{x_n})(x').
	$$
	For $i=1,\dots,n-1$,
	$$
	\partial_{x_i}\varphi_j
	=
	x_n\theta(x_n)\,((\partial_{x_i}g_j)*\rho_{x_n}).
	$$

	For $i=n$,
	\begin{eqnarray*}
		\partial_{x_n}\varphi_j
		&=&\left[ \theta(x_n)+x_n\theta'(x_n)\right] (g_j\ast \rho_{x_n})+ x_n\theta(x_n)\sum_{k=1}^{n-1} g_j\ast \partial_{x_k}(x_k \rho)_{x_n}\\&=& \left[ \theta(x_n)+x_n\theta'(x_n)\right] (g_j\ast \rho_{x_n})+ x_n\theta(x_n)\sum_{k=1}^{n-1}\partial_{x_k} g_j\ast (x_k \rho)_{x_n}
	\end{eqnarray*}

	We note that Lemma \ref{lem:convholder} applies to  $\rho$ and $x_k \rho$. This allows us to control the corresponding convolutions in $C^{0,\beta}$ uniformly with respect to $x_n$. Hence,
$$\|\partial_{x_i}\varphi_j\|_{C^{0,\beta}(Q_1^+)}\leq 	C\|\widetilde \u\|_{C^{1,\alpha}(Q_1^+)}$$	
	
	\medskip
	\noindent
	\textbf{ Tangential second derivatives.}
	
	For $i,k\in\{1,\dots,n-1\}$,
	$$
	\partial_{x_k}\partial_{x_i}\varphi_j
	=
	x_n\theta(x_n)\,
	\partial_{x_k}\bigl((\partial_{x_i}g_j)*\rho_{x_n}\bigr)
	=
	x_n\theta(x_n)\,
	\bigl((\partial_{x_i}g_j)*(\partial_{x_k}\rho)_{x_n}\bigr).
	$$
	Since $\displaystyle\int_{\mathbb{R}^{n-1}} \partial_{x_k}\rho =0$, Lemma \ref{lem:convholder} applies to  $\partial_{x_k}\rho$, and yields
	$$
	\|\partial_{x_k}\partial_{x_i}\varphi_j\|_{C^{0,\beta}(Q_1^+)}
	\le
	C\|\partial_{x_i}g_j\|_{C^{0,\alpha}(Q_1^+)}
	\le
	C\|\widetilde \u\|_{C^{1,\alpha}(Q_1^+)}.
	$$
	
	\medskip
	\noindent
	\textbf{ Mixed derivatives.}
	
	Differentiating $\partial_{x_i}\varphi_j$ with respect to $x_n$, we obtain
	
	\begin{eqnarray*}
	\partial_{x_n}	\partial_{x_i} \varphi_j =\left[x_n \theta'(x_n)+(2-n)\theta(x_n) \right]\left(\partial_{x_i}g_j \ast \rho_{x_n} \right)  -\theta(x_n)\left( \partial_{x_i}g_j\ast \psi_{x_n}\right) 
	\end{eqnarray*}
	
	where $\psi(x'):=\displaystyle\sum_{i=1}^{n-1} x_i \partial_{x_i} \rho(x')$, $\displaystyle\int_{\r^{n-1}}\psi(x')\,dx'=-(n-1)\int_{\r^{n-1}}\rho(x')\,dx'=-(n-1).$

	By Lemma \ref{lem:convholder}, both terms are uniformly bounded in $C^{0,\beta}$.
	Therefore,
	$$
	\|\partial_{x_n}\partial_{x_i}\varphi_j\|_{C^{0,\beta}(Q_1^+)}
	\le
	C\|\widetilde \u\|_{C^{1,\alpha}(Q_1^+)}.
	$$
	
	\medskip
	\noindent
	\textbf{Second normal derivatives.}
	
	\begin{eqnarray*}
		\partial_{x_n}\partial_{x_n} \varphi_j&=&\left[2\theta'(x_n)+x_n\theta''(x_n) \right](g_j \ast \rho_{x_n})+ (x_n \theta(x_n))' \sum_{i=1}^{n-1}\partial_{x_i}g_j\ast (x_i\rho)_{x_n} \\&+& x_n \theta(x_n)\sum_{i=1}^{n-1} \partial_{x_i}g_j \ast \partial_{x_n} (x_j \rho)_{x_n}.
	\end{eqnarray*}

	Hence, Lemma \ref{lem:convholder} applies to each term. Consequently,
	\[
	\|\partial_{x_n}\partial_{x_n}\varphi_j\|_{C^{0,\beta}(Q_1^+)}
	\le
	C\|\widetilde \u\|_{C^{1,\alpha}(Q_1^+)}.
	\]
	
	Combining the previous estimates, we get
	\[
	\varphi_j\in C^{2,\beta}(Q_1^+),
	\qquad
	\|\varphi_j\|_{C^{2,\beta}(Q_1^+)}
	\le
	C\|\widetilde \u\|_{C^{1,\alpha}(Q_1^+)}.
	\]

	\medskip
	\noindent

	Since each $\varphi_j\in C^{2,\beta}(Q_1^+)$, it follows that
	$
	\Psi\in C^{1,\beta}(Q_1^+)^n,
	$
	and
	$$
	\|\Psi\|_{C^{1,\beta}(Q_1^+)}
	\le
	C\sum_{j=1}^{n-1}\|\varphi_j\|_{C^{2,\beta}(Q_1^+)}
	\le
	C\|\widetilde \u\|_{C^{1,\alpha}(Q_1^+)}.
	$$
	
	Moreover,
	$$
	\operatorname{div}\Psi
	=
	\sum_{j=1}^{n-1}\partial_{x_j}\partial_{x_n}\varphi_j
	+
	\partial_{x_n}\!\left(-\sum_{j=1}^{n-1}\partial_{x_j}\varphi_j\right)
	=
	0,
	$$
	by equality of mixed derivatives.
	
	Finally, since $\theta(x_n)=0$ for $x_n\ge \frac12$, we have
	\[
	\Psi\equiv 0 \qquad \text{for } x_n\ge \frac12.
	\]
	
For the boundary values, we have from the formula for $\partial_{x_n}\varphi_j$, $j=1, \dots, n-1$,
	$$
	\Psi_j(x',x_n)
	=\left[ \theta(x_n)+x_n\theta'(x_n)\right] (g_j\ast \rho_{x_n})(x')+ x_n\theta(x_n)\sum_{k=1}^{n-1} (g_j\ast \partial_{x_k}(x_k \rho)_{x_n})(x')
	$$
	Since $\displaystyle\int \rho=1$, Lemma \ref{lem:convholder} gives
	$$
	g_j*\rho_{x_n}\to g_j
	\quad\text{in } C^{0,\beta}
	\quad\text{as } x_n\to 0^+.
	$$
	and by the compact support of $\rho$
	$$\int_{\mathbb{R}^{n-1}} \partial_k(z_k\rho(z'))\,dz'
	=
	0,
	$$
	so Lemma \ref{lem:convholder} yields
	 $$(g_j\ast \partial_{x_k}(x_k \rho)_{x_n})(x')
      \to 0
	\quad\text{in } C^{0,\beta}
	\quad\text{as } x_n\to 0^+.
	$$
	Therefore,
	$$
	\Psi_j(\cdot,x_n)\to g_j=\widetilde u_j(\cdot,0)
	\quad\text{in } C^{0,\beta}
	\quad\text{as } x_n\to 0^+,
	$$
	that is,
	$$
	\Psi_j(x',0)=\widetilde u_j(x',0),
	\qquad j=1,\dots,n-1.
	$$
	
	On the other hand,
	$$
	\Psi_n(x',x_n)
	=
	-\sum_{j=1}^{n-1}\partial_{x_j}\varphi_j(x',x_n)
	=
	-x_n\theta(x_n)\sum_{j=1}^{n-1}
	\bigl((\partial_{x_j}g_j)*\rho_{x_n}\bigr)(x').
	$$
	By Lemma \ref{lem:convholder}, the convolutions are uniformly bounded, hence
	$$
	|\Psi_n(x',x_n)|
	\le
	C x_n \|\widetilde \u\|_{C^{1,\alpha}(Q_1^+)},
	$$
		this estimate shows that $\Psi_n(\cdot,x_n) \to 0$ uniformly as $x_n \to 0^+$. By continuity up to the boundary, this implies $\Psi_n(x',0) = 0$.
	\medskip
	\noindent

\end{proof}

\begin{remark}
	In the previous proof, the boundary identities for $\Psi$ are first obtained by taking the limit as $x_n \to 0^+$ with $x'$ fixed. However, since $\Psi \in C^{1,\beta}(Q^+_1)^n$, it admits a continuous extension to the closure $\overline{Q^+_1}$. In particular, the trace of $\Psi$ on $\{x_n = 0\}$ is well defined and coincides with the limit from within the domain. Therefore, the boundary values obtained above hold independently of the way the boundary is approached.
\end{remark}

\section{Proof of the Theorem \ref{thm:global}: Construction of the global solution}

Let $\Omega \subset \mathbb{R}^n$ be a bounded $C^2$ domain and
let $f \in C^{0,\alpha}(\Omega)$ satisfying $\displaystyle \int_\Omega f = 0.$ There exists a finite open covering $\{U_i\}_{i=1}^N$ of $\overline{\Omega}$ such that:

\begin{itemize}
	\item either $U_i \subset \subset \Omega$,
	\item or $U_i$ intersects $\partial\Omega$ and the boundary can be written locally as a graph in $U_i$.
\end{itemize}

Assume that $U_i$ intersects the boundary. Since $\Omega$ is a $C^2$ domain, we may choose the covering neighborhoods near the boundary so that, after a suitable translation, rotation,  and possibly shrinking $U_i$,
$$
\Omega \cap U_i
=
\{(x',x_n)\in Q : x_n>\gamma(x')\},
$$
where $\gamma\in C^2$ and $Q$ is a cube centered at the origin.
We choose a partition of unity
$\{\phi_i\}_{i=1}^N \subset C_0^\infty(U_i)$ such that $\displaystyle \sum_{i=1}^N \phi_i = 1
\quad \text{in } \overline{\Omega}.$ Define $f_i := \phi_i f$. Since $\phi_i \in C_0^\infty(U_i)$, $\mbox{ supp } f_i \subset \subset U_i$, moreover,
$
f = \displaystyle\sum_{i=1}^N f_i.
$
Although $\displaystyle\int_\Omega f = 0$, in general
$
\displaystyle\int_\Omega f_i \neq 0.
$
Then, for each $i$ choose $\psi_i \in C_0^\infty(U_i \cap \Omega)$ such that $\displaystyle \int_\Omega \psi_i = 1.$ Define $\widetilde f_i:=f_i-\left(\displaystyle \int_\Omega f_i \right)\psi_i$, then 
$$
\int_\Omega \widetilde f_i = 0,
\qquad
\mbox{ supp } \widetilde f_i \subset \subset U_i.
$$

Since $\displaystyle\sum_{i=1}^N \int_\Omega f_i = 0$, we still have $f = \displaystyle\sum_{i=1}^N \widetilde f_i.$ To simplify notation, we relabel $\widetilde f_i$ as $f_i$.

If $U_i \subset \subset \Omega$, then $f_i$ has compact support strictly contained in $\Omega$. Extending $f_i$ by zero outside $\Omega$, we may apply the compact support solvability result from \cite{CD}. Therefore, there exists
$$
\u_i \in C^{1,\alpha}(\mathbb{R}^n)^n
$$
such that
$$
\operatorname{div}\u_i=f_i
\quad\text{in }\mathbb{R}^n,
$$
$\u_i$ vanishes outside $\Omega$, and
$$
\|\u_i\|_{C^{1,\alpha}(\mathbb{R}^n)}
\le
C\|f_i\|_{C^{0,\alpha}(\Omega)}.
$$
Restricting to $U_i$, we obtain the desired estimates.

Assume now that $U_i$ intersects the boundary. Define the flattening map
$$
\Phi(x',x_n)=(y',y_n)=(x',\, x_n - \gamma(x')).
$$

\[
D\Phi(x)
=
\begin{pmatrix}
	I_{n-1} & 0 \\
	-\nabla\gamma(x') & 1
\end{pmatrix}
\]

\[
D\Phi^{-1}(y)
=
\begin{pmatrix}
	I_{n-1} & 0 \\
	\nabla\gamma(y') & 1
\end{pmatrix}.
\]

Then $\det D\Phi =1$. Moreover, by construction, $\Phi(U_i\cap\Omega)=Q^+$.

 Define
$F_i(y):=f_i(\Phi^{-1}(y))$, $y\in Q^+$. By the change of variables we have
$$
\int_{Q^+} F_i(y)\,dy
=
\int_{\Omega\cap U_i} f_i(x)\,dx
=0.
$$

By Theorem \ref{thm:half-cube},  there exists a vector field
$
\vv_i\in C^{1,\beta}(Q^+)^n
$
such that
$$
\operatorname{div}_y \vv_i = F_i \quad \text{in } Q^+,
$$

$$
\vv_i=0 \quad \text{on } \partial Q^+
$$
for $0<\beta<\alpha$ and there exists a constant $C=C(\alpha,\beta,n,Q^+)$ so that
\begin{equation}\label{eq:semi-beta-est}
	\|\vv_i\|_{C^{1,\beta}(Q^+)}
	\le
	C \|F_i\|_{C^{0,\alpha}(Q^+)}.
\end{equation}

We now define the vector field in the original variables by
$$
\u_i(x):=D\Phi^{-1}(y)\,\vv_i(y),
\qquad y=\Phi(x).
$$
Its components are given by
$$
u_{i,\ell}(x)=v_{i,\ell}(y), \qquad \ell=1,\dots,n-1,
$$
and
$$
u_{i,n}(x)=v_{i,n}(y)+\nabla\gamma(y')\cdot \bigl(v_{i,1}(y),\dots,v_{i,n-1}(y)\bigr),
$$

We claim that
$$
\operatorname{div}_x \u_i = f_i
\qquad \text{in } U_i\cap\Omega.
$$
Indeed, for $\ell=1,\dots,n-1$,
$$
\partial_{x_\ell}
=
\partial_{y_\ell}
-
(\partial_{y_\ell}\gamma)\,\partial_{y_n},
\qquad
\partial_{x_n}=\partial_{y_n}.
$$
Hence
$$
\sum_{\ell=1}^{n-1}\partial_{x_\ell}u_{i,\ell}
=
\sum_{\ell=1}^{n-1}
\Bigl(
\partial_{y_\ell}v_{i,\ell}
-
(\partial_{y_\ell}\gamma)\,\partial_{y_n}v_{i,\ell}
\Bigr).
$$
On the other hand,
$$
\partial_{x_n}u_{i,n}
=
\partial_{y_n}
\Bigl(
v_{i,n}
+
\nabla\gamma(y')\cdot \bigl(v_{i,1}(y),\dots,v_{i,n-1}(y)\bigr)
\Bigr)
=
\partial_{y_n}v_{i,n}
+
\sum_{\ell=1}^{n-1}
(\partial_{y_\ell}\gamma)\,\partial_{y_n}v_{i,\ell},
$$
because $\gamma$ depends only on $y'$.
Summing the two identities, the mixed terms cancel and we obtain
$$
\operatorname{div}_x \u_i
=
\sum_{\ell=1}^{n-1}\partial_{y_\ell}v_{i,\ell}
+
\partial_{y_n}v_{i,n}
=
\operatorname{div}_y \vv_i
=
F_i(y)
=
f_i(x).
$$

It remains to establish the boundary regularity. Since
$\vv_i\in C^{1,\beta}(Q^+)^n$
and $\Phi$ and $\Phi^{-1}$ are $C^2$ maps, we  obtain by composition that
$$
\u_i\in C^{1,\beta}(U_i\cap\Omega)^n
\qquad \text{for every } 0<\beta<\alpha.
$$

Since $\Phi$ and $\Phi^{-1}$ are $C^2$ diffeomorphisms, standard composition estimates yield
$$
\|\u_i\|_{C^{1,\beta}(U_i\cap\Omega)}
\le
C
\|\vv_i\|_{C^{1,\beta}(Q^+)}.
$$
 Using \eqref{eq:semi-beta-est}, we infer
\[
\|\u_i\|_{C^{1,\beta}(U_i\cap\Omega)}
\le
C \|f_i\|_{C^{0,\alpha}(U_i\cap\Omega)}.
\]

It remains to verify the boundary condition. Given that $\vv_i=0$ on $\partial Q^+$,
and the flattening map sends
$$
\partial\Omega\cap U_i
\quad \text{onto} \quad
\partial Q^+ \cap \left\lbrace y_n=0 \right\rbrace ,
$$
we obtain
$$
\u_i(x)=D\Phi^{-1}(y)\,\vv_i(y)=0
\qquad \text{for every } x\in \partial\Omega\cap U_i.
$$

Finally, define
$$
\u:=\sum_{i=1}^N \u_i.
$$
Since the covering is finite, we conclude that $\u\in C^{1, \beta}(\Omega)^n$. Moreover,
$$
\operatorname{div}\u
=
\sum_{i=1}^N \operatorname{div}\u_i
=
\sum_{i=1}^N f_i
=
f
\quad \text{in } \Omega.
$$
Since each local solution vanishes on the corresponding boundary patch, we also have
$$
\u=0 \qquad \text{on } \partial\Omega.
$$

and
$$
\|\u\|_{C^{1,\beta}(\Omega)}
\le
C \|f\|_{C^{0,\alpha}(\Omega)},
\qquad 0<\beta<\alpha.
$$

This completes the proof.

\begin{remark}
We do not know whether the regularity assumption on the boundary can be weakened. In the present approach, the flattening procedure requires differentiating the transformed vector field
$$
\u(x)=D\Phi^{-1}(\Phi(x))\bf{v}(\Phi(x)),
$$
which introduces second derivatives of the boundary defining function through the term $D(D\Phi^{-1})$. Consequently, the proof relies on the $C^2$ regularity of the flattening map. Whether the result remains valid under weaker boundary assumptions, such as $C^{1,\alpha}$ regularity, would require a different argument and is left open.
	
\end{remark}

\section*{Acknowledgments}
This paper is dedicated to the memory of Ricardo G. Durán.

\end{document}